\documentclass[11pt,reqno]{amsart}

\usepackage[T1]{fontenc}
\usepackage{lmodern}
\usepackage{amsmath,amssymb,amsthm,mathtools}
\usepackage{esint}
\usepackage{microtype}
\usepackage{cite}
\usepackage[hidelinks]{hyperref}
\usepackage[a4paper,margin=1in]{geometry}

\numberwithin{equation}{section}

\newtheorem{theorem}{Theorem}[section]
\newtheorem{lemma}[theorem]{Lemma}
\newtheorem{proposition}[theorem]{Proposition}

\newcommand{\Sph}{\mathbb S}
\newcommand{\tr}{\operatorname{tr}}
\newcommand{\Vol}{\operatorname{Vol}}

\newcommand{\Ric}{\operatorname{Ric}}
\newcommand{\diver}{\operatorname{div}}
\newcommand{\ip}[2]{\langle #1,#2\rangle}
\newcommand{\norm}[1]{\lvert #1\rvert}

\title[The second gap with constant cubic trace]{The Second Gap for Minimal Hypersurfaces in the Unit Sphere with Constant Cubic Trace}
\author[S. Guan]{Shanlin Guan}
\address{School of Mathematical Sciences, Chongqing Normal University,
Chongqing 401331, P. R. of China}
\email{20269003@cqnu.edu.cn}
\date{}
\hypersetup{pdftitle={The Second Gap for Minimal Hypersurfaces in the Unit Sphere with Constant Cubic Trace},pdfauthor={Shanlin Guan}}

\begin{document}

\begin{abstract}
Let $M^n$, $n\ge3$, be a closed oriented minimal hypersurface in the unit sphere with constant scalar curvature and constant $f_3=\tr(h^3)$, where $h$ is the second fundamental form. We prove that $S=\norm h^2>n$ implies $S\ge2n$. Equality holds if and only if each connected component is locally congruent to a Cartan minimal isoparametric hypersurface with three distinct principal curvatures.
\end{abstract}

\maketitle

\section{Introduction}

Let $M^n$ be a closed oriented minimal hypersurface immersed in $\Sph^{n+1}(1)$, and let $h$ denote its second fundamental form with respect to a fixed global unit normal. Identifying $h$ with the corresponding shape operator, we write
\[
S=\norm h^2,\qquad f_r=\tr(h^r).
\]
The Gauss equation expresses the scalar curvature as $n(n-1)-S$, so constancy of the scalar curvature is equivalent to constancy of $S$. In its hypersurface form, Chern's conjecture asserts that, for fixed $n$, the possible constant values of $S$ form a discrete set; see \cite{Yau1982}. The stronger form of the conjecture asserts that every such hypersurface is isoparametric. The minimal isoparametric hypersurfaces satisfy $S=(g-1)n$, where the number $g$ of distinct principal curvatures belongs to $\{1,2,3,4,6\}$ by M\"unzner's theorem \cite{Munzner1980,Munzner1981}. Thus their possible values of $S$ are $0,n,2n,3n$, and $5n$, with the usual restrictions on the multiplicities. For the relation with isoparametric geometry, see \cite{Cecil2008,GeTang2012,ScherfnerWeissYau2012}.

Simons' identity excludes the interval $0<S<n$ \cite{Simons}, and the case $S=n$ is characterized by the Clifford minimal hypersurfaces \cite{CdCK,Lawson}. The next expected gap is $(n,2n)$. Peng and Terng established a positive gap above $n$ and proved the sharp implication $S>3\Rightarrow S\ge6$ in dimension three \cite{PTseminar,PTannalen}; Chang subsequently completed the three-dimensional classification \cite{Chang1993}. In higher dimensions, successive refinements of the Peng--Terng estimates have enlarged the excluded interval. Under the assumption $S>n$, Yang and Cheng obtained $S\ge4n/3$ \cite{YangCheng}, Suh and Yang obtained $S>10n/7$ \cite{SuhYang}, and a recent result of Chen gives $S>85n/59$ for $n\ge4$ \cite{Chen2026}. For further pinching results, see \cite{YangCheng1994,WeiXu2007,Zhang2010,DingXin2011,XuXu2017}.

The assumption that $f_3$ is constant is suggested by the hypotheses themselves: minimality gives $f_1=0$, while constant scalar curvature fixes $f_2=S$, leaving $f_3$ as the first power sum not prescribed. With this additional assumption, Yang and Cheng proved, for $n\ge4$, that $S>5n/3$ whenever $S>n$ \cite{YangCheng}. Cheng, Wei and Yamashiro improved the bound to $S>1.8252n-0.712898$ for $n\ge5$, under the weaker assumption of completeness \cite{CWY}. In the minimal Willmore case, where $f_3=0$, Ge, Tan, Yan and Zhang obtained an approximate second gap and proved the exact bound $S\ge2n$ under an additional lower bound for $f_4$ \cite{GeTanYanZhang}. For general $n$, these estimates still leave part of $(n,2n)$ open when only constancy of $f_3$ is assumed.

Prescribing further power sums leads to isoparametric rigidity theorems. De Almeida and Brito proved that a closed three-dimensional hypersurface with constant mean curvature and constant nonnegative scalar curvature is isoparametric \cite{deAlmeidaBrito}. Tang, Wei and Yan extended this conclusion to higher dimensions under constancy of $f_1,\ldots,f_{n-1}$, nonnegative total scalar curvature, and the assumption that all principal curvatures are distinct at every point \cite{TangWeiYan2020}. Tang and Yan removed the last assumption: a closed hypersurface of dimension $n\ge4$ is isoparametric if its scalar curvature is nonnegative and $f_1,\ldots,f_{n-1}$ are constant \cite{TangYan2023}. For minimal hypersurfaces of dimension four, these conditions on the power sums amount to constancy of $S$ and $f_3$. In higher dimensions, however, prescribing $f_3$ still leaves the higher power sums undetermined.

In dimension four, Deng, Gu and Wei proved isoparametricity in the minimal Willmore case \cite{DengGuWei}. For constant $f_3$, Li obtained rigidity under additional bounds on $f_3$ and the Gauss--Kronecker curvature \cite{Li2022}. Spruck and Xiao proved that a closed minimal hypersurface with constant $S$ and constant $f_3$ is isoparametric under the additional assumption of nonnegative scalar curvature, and hence $S\in \{0,4,12\}$ \cite{SpruckXiao2025}. He, Xu and Zhao subsequently removed the nonnegative scalar curvature assumption and obtained the same classification under constant $S$ and constant $f_3$ alone \cite{HeXuZhao}. For each fixed dimension, Tan, Tang, Xie and Yan proved that the set of constant values of $S$ is locally finite for closed embedded minimal hypersurfaces with constant $f_3$ \cite{TanTangXieYan}. This does not determine the location of the second gap. For immersed hypersurfaces, we prove the following theorem.

\begin{theorem}\label{thm:main}
Let $M^n$, $n\ge3$, be a closed oriented hypersurface minimally immersed in $\Sph^{n+1}(1)$, with second fundamental form $h$ relative to a global unit normal. Suppose that $S=\norm h^2$ and $f_3=\tr(h^3)$ are constant. If $S>n$, then
\[
S\ge2n.
\]
Equality holds if and only if each connected component of $M$ is locally congruent to a Cartan minimal isoparametric hypersurface with three distinct principal curvatures. In particular, equality is possible only for $n\in\{3,6,12,24\}$.
\end{theorem}

Thus constancy of $f_3$ suffices for the second gap, with no further power sums prescribed. The new conclusion concerns dimensions $n\ge5$; in dimensions three and four, stronger classification results are already known. At $S=2n$, the proof also gives $f_3=0$ and $h^3=3h$, from which the rigidity statement follows by Cartan's classification \cite{Cartan1939,Cecil2008}.

To reach the value $2n$, we retain the variation of $f_4$ in the integral estimates. Consider
\[
\Phi=h^2-\frac{f_3}{S}h-\frac Sn I,
\]
where $I$ is the identity endomorphism. This is the component of $h^2$ orthogonal to $I$ and $h$; the same tensor
is already used in Cheng, Wei and Yamashiro \cite{CWY}, its squared norm
occurs in the moment estimates of \cite{YangCheng,CWY}, and the tensor is
also used in \cite{Chen2026}. Constancy of $f_3$ allows us to choose a divergence-free cubic polynomial in $h$ whose integrated Weitzenb\"ock estimate contains a variance term. The scalar identity for $\norm\Phi^2$ cancels this term and controls the trace contributions in the fourth-order estimate. Taking trace-free parts then permits a comparison with the endpoint $S=2n$, which excludes the whole interval $n<S<2n$.

\section{The quadratic defect and the cubic estimate}\label{sec:defect}

Throughout, $M$ satisfies the hypotheses of Theorem~\ref{thm:main}. Here closed means compact without boundary. All arguments apply componentwise. Our curvature convention is
\[
R_{ijkl}=\delta_{ik}\delta_{jl}-\delta_{il}\delta_{jk}
+h_{ik}h_{jl}-h_{il}h_{jk}.
\]
We take $\Delta=\tr\nabla^2$ and use the metric inner products on tensors, with full sums over ordered indices. Indices range from $1$ to $n$, and repeated indices are summed unless a summation sign is displayed. In a local orthonormal frame $\{e_i\}$, commas indicate covariant differentiation:
\[
h_{ij,k}=(\nabla h)(e_i,e_j,e_k),
\qquad
h_{ij,kl}=(\nabla^2h)(e_i,e_j,e_k,e_l),
\]
with the last index indicating the last derivative. We also write $h_{ijk}=h_{ij,k}$ and $h_{ijkl}=h_{ij,kl}$, and use the same differentiation convention for other tensors. Codazzi implies that $h_{ijk}$ is completely symmetric. Integrals without a specified measure are taken with respect to the induced volume form $d\mu$. We write
\[
\fint_M u:=\frac1{\Vol(M)}\int_Mu\,d\mu
\]
for the normalized integral. Reversing the chosen normal changes the signs of $h$ and $f_3$ but leaves all scalar estimates below unchanged.

We differentiate covariantly before choosing a frame that diagonalizes $h$. At a point where $h_{ij}=\lambda_i\delta_{ij}$, set
\[
\mathfrak A=\sum_{i,j,k}\lambda_i^2h_{ijk}^2,
\qquad
\mathfrak B=\sum_{i,j,k}\lambda_i\lambda_jh_{ijk}^2,
\qquad
\mathfrak C=\sum_{i,j,k}\lambda_i h_{ijk}^2.
\]
For constant $S$ and $f_3$, the Simons and Peng--Terng identities take the form \cite{Simons,PTseminar,PTannalen,YangCheng}:
\begin{align}
\norm{\nabla h}^2&=S(S-n),\label{eq:grad-h}\\
\mathfrak C&=\frac{S-n}{2}f_3,\label{eq:C-id}\\
\mathfrak A-2\mathfrak B&=Sf_4-f_3^2-S^2,\label{eq:A-2B}\\
\frac14\Delta f_4&=(n-S)f_4+2\mathfrak A+\mathfrak B.
\label{eq:lap-f4}
\end{align}
The first two follow from $\Delta h=(n-S)h$ and
$\Delta f_3/3=(n-S)f_3+2\mathfrak C$. Both constancy assumptions enter
\eqref{eq:A-2B}. Ricci commutation gives
\begin{equation}\label{eq:Ricci-commutation}
h_{ij,kl}-h_{ij,lk}
=h_{ip}R_{pjkl}+h_{pj}R_{pikl}.
\end{equation}
Differentiating $S$ and $f_3$ twice and then choosing a principal frame gives, for each $k$,
\[
\sum_i\lambda_i h_{ii,kk}=-\sum_{i,j}h_{ijk}^2,
\qquad
\sum_i\lambda_i^2h_{ii,kk}=-2\sum_{i,j}\lambda_i h_{ijk}^2.
\]
Consequently,
\[
\sum_{i,k}\lambda_i\lambda_k^2h_{ii,kk}=-\mathfrak A,
\qquad
\sum_{i,k}\lambda_i^2\lambda_kh_{ii,kk}=-2\mathfrak B.
\]
Codazzi and \eqref{eq:Ricci-commutation} imply
$h_{ii,kk}-h_{kk,ii}=(\lambda_i-\lambda_k)(1+\lambda_i\lambda_k)$.
Interchanging $i,k$ in the first contraction and subtracting, we obtain
\[
\mathfrak A-2\mathfrak B
=\sum_{i,k}\lambda_i^2\lambda_k(\lambda_i-\lambda_k)(1+\lambda_i\lambda_k)
=Sf_4-f_3^2-S^2,
\]
which proves \eqref{eq:A-2B}.

From now on assume $S>n$, and write
\begin{equation}\label{eq:basic-variables}
a=\frac{f_3}{S},\qquad
\Phi=h^2-ah-\frac Sn I,
\qquad
p=\frac{\norm\Phi^2}{S^2},
\qquad
z=\frac{a^2}{S}=\frac{f_3^2}{S^3},
\end{equation}
\[
t=\frac{S-n}{S},\qquad d=3-4t,
\qquad x=\fint_Mp.
\]
Here $a,z,t,d$ are constant and $p$ is nonnegative, with average $x$. Orthogonality of $\Phi$ to $I$ and $h$ gives
\begin{equation}\label{eq:p-formula}
p=\frac{f_4}{S^2}-z-\frac1n.
\end{equation}
For $r\in\mathbb R$ define
\begin{equation}\label{eq:omega-r}
\omega_r=
\frac1{3S^3}\sum_{i,j,k}
\left(\lambda_i+\lambda_j+\lambda_k-ra\right)^2h_{ijk}^2,
\qquad
\kappa(r)=4-5r+\frac53r^2.
\end{equation}

\begin{lemma}\label{lem:scalar}
For every $r\in\mathbb R$,
\begin{equation}\label{eq:scalar-defect}
\frac1S\Delta p
=-t\left(\frac1n+\kappa(r)z\right)+dp+5\omega_r.
\end{equation}
Consequently,
\begin{align}
5\fint_M\omega_r
&=t\left(\frac1n+\kappa(r)z\right)-dx,\label{eq:omega-average}\\
5\fint_M(p-x)\omega_r
&=-d\fint_M(p-x)^2-\frac1S\fint_M\norm{\nabla p}^2.
\label{eq:omega-cov}
\end{align}
Moreover, the following bound holds pointwise on $M$:
\begin{equation}\label{eq:moment-region}
2p+z\le1-\frac2n.
\end{equation}
\end{lemma}

\begin{proof}
The complete symmetry of $h_{ijk}$ gives
\[
S^3\omega_r
=\mathfrak A+2\mathfrak B-2ra\mathfrak C
+\frac{r^2a^2}{3}\norm{\nabla h}^2.
\]
Using \eqref{eq:grad-h}--\eqref{eq:A-2B}, we obtain
\[
\frac{\mathfrak A+2\mathfrak B}{S^3}
=\omega_r+\left(r-\frac{r^2}{3}\right)tz,
\qquad
\frac{\mathfrak A-2\mathfrak B}{S^3}=p+\frac tn.
\]
Since $4(2\mathfrak A+\mathfrak B)=5(\mathfrak A+2\mathfrak B)+3(\mathfrak A-2\mathfrak B)$,
\eqref{eq:lap-f4} and \eqref{eq:p-formula} give
\begin{align*}
\frac1S\Delta p
&=-4t\left(p+z+\frac1n\right)
 +5\omega_r+\left(5r-\frac53r^2\right)tz+3p+\frac{3t}{n}\\
&=dp-t\left[\frac1n+\left(4-5r+\frac53r^2\right)z\right]+5\omega_r.
\end{align*}
Integrating proves \eqref{eq:omega-average}. Multiplication by $p-x$ and integration by parts yield \eqref{eq:omega-cov}, since $\fint_M(p-x)=0$.

For the pointwise estimate, put $\mu_i=\lambda_i/\sqrt S$ and
$m_j=\sum_i\mu_i^j$. Thus $m_1=0$, $m_2=1$, $z=m_3^2$, and
$p=m_4-m_3^2-1/n$. The identity
\[
1-2m_4+m_3^2
=\sum_{i<j}\mu_i^2\mu_j^2(\mu_i+\mu_j)^2
+6\sum_{i<j<k}\mu_i^2\mu_j^2\mu_k^2\ge0
\]
follows by expanding the two sums as $m_4+m_3^2-2m_6$ and $1-3m_4+2m_6$, respectively; see also \cite[Lemma~2.3]{Chen2026}. It is equivalent to \eqref{eq:moment-region}.
\end{proof}
In the range $n<S\le2n$, one has $0<t\le1/2$ and $d\ge1$. Taking $r=3/2$ in \eqref{eq:omega-average} and using $\kappa(3/2)=1/4$ and $\omega_{3/2}\ge0$ therefore gives
\begin{equation}\label{eq:scalar-admissible}
x\le\frac td\left(\frac1n+\frac z4\right).
\end{equation}
The cubic estimate uses $r=5/4$. Since $\kappa(5/4)=17/48$, we set
\begin{equation}\label{eq:A54}
A_{5/4}:=
 t\left(\frac1n+\frac{17z}{48}\right)-dx
=5\fint_M\omega_{5/4}\ge0.
\end{equation}
Consider the cubic tensor
\[
\Theta=\left(h-\frac a4I\right)\Phi-Sxh,
\qquad T=\fint_M\frac{\norm\Theta^2}{S^3}.
\]
Subtracting $Sxh$ ensures that $\fint_M\ip h\Theta=0$. To estimate $T$, we use the covariant exterior derivative, normalized by
\begin{equation}\label{eq:dnabla}
(d^\nabla\Theta)_{kij}=\Theta_{ij,k}-\Theta_{kj,i}.
\end{equation}

\begin{proposition}\label{prop:cubic}
For $n<S\le2n$ one has
\begin{equation}\label{eq:T-bound}
T\le K A_{5/4},
\end{equation}
where
\begin{equation}\label{eq:Kdef}
K=
\frac1{1-t}
\left[
\frac xd+\frac4{15n}
+\frac{(2-\frac4n)d}{225}
+\frac z{15+d}
\right].
\end{equation}
\end{proposition}

\begin{proof}
The tensor $\Theta$ is trace-free, as $\Phi\perp I,h$. Its expansion is
\begin{equation}\label{eq:Theta-expanded}
\Theta=h^3-\frac{5a}{4}h^2
+\left(\frac{a^2}{4}-\frac Sn-Sx\right)h
+\frac{aS}{4n}I.
\end{equation}
Thus $\Theta$ commutes with $h$. Minimality and Codazzi give $\diver h=0$ and
\[
\diver(h^2)=\frac12\nabla S,
\qquad
\diver(h^3)=\frac12h(\nabla S)+\frac13\nabla f_3.
\]
Since $S$, $f_3$, and the coefficients in \eqref{eq:Theta-expanded} are constant, both $\Phi$ and $\Theta$ are divergence-free.
The Ricci identity \eqref{eq:Ricci-commutation} for $\Theta$ reads
\[
\Theta_{ij,kl}-\Theta_{ij,lk}
=\Theta_{ip}R_{pjkl}+\Theta_{pj}R_{pikl}.
\]
From \eqref{eq:dnabla} and the symmetry of $\Theta$,
\[
\frac12\norm{d^\nabla\Theta}^2
=\norm{\nabla\Theta}^2
-\sum_{i,j,k}\Theta_{ij,k}\Theta_{kj,i}.
\]
Integrating the mixed term by parts, commuting the last two derivatives, and using $\diver\Theta=0$ yields
\[
\int_M \Theta_{ij,k}\Theta_{ik,j}
=-\int_M\Ric_{ik}\Theta_{ij}\Theta_{kj}
+\int_M R_{ikjl}\Theta_{ij}\Theta_{kl}.
\]
As $h$ and $\Theta$ commute, we may diagonalize them in the same frame, writing
$h_{ij}=\lambda_i\delta_{ij}$ and $\Theta_{ij}=\theta_i\delta_{ij}$. The Gauss equation and $\sum_i\theta_i=0$ give
\begin{align*}
&\Ric_{ik}\Theta_{ij}\Theta_{kj}-R_{ikjl}\Theta_{ij}\Theta_{kl}\\
&\qquad=\sum_i(n-1-\lambda_i^2)\theta_i^2
-\sum_{i\ne j}(1+\lambda_i\lambda_j)\theta_i\theta_j\\
&\qquad=n\norm\Theta^2-\ip{h}{\Theta}^2.
\end{align*}
Hence
\begin{equation}\label{eq:Weitzenbock}
\frac12\fint_M\norm{d^\nabla\Theta}^2
=\fint_M\left(\norm{\nabla\Theta}^2+n\norm\Theta^2-\ip{h}{\Theta}^2\right).
\end{equation}
The completely symmetric part of $\nabla\Theta$ is
\[
(\nabla\Theta)^{\mathrm{sym}}_{ijk}
=\frac13(\Theta_{ij,k}+\Theta_{ik,j}+\Theta_{jk,i}).
\]
The orthogonal decomposition gives
\[
\norm{\nabla\Theta}^2
=\norm{(\nabla\Theta)^{\mathrm{sym}}}^2
+\frac13\norm{d^\nabla\Theta}^2.
\]
Also, since $h^2=\Phi+ah+(S/n)I$,
\begin{equation}\label{eq:hTheta}
\ip h\Theta=\ip{h^2}{\Phi}-S^2x=S^2(p-x).
\end{equation}
Substituting in \eqref{eq:Weitzenbock}, dividing by $S^4$, and discarding the nonnegative symmetric term, we obtain
\begin{equation}\label{eq:T-Weitz-final}
(1-t)T\le
\fint_M\frac{\norm{d^\nabla\Theta}^2}{6S^4}
+\fint_M(p-x)^2.
\end{equation}
Differentiating \eqref{eq:Theta-expanded} and then choosing a principal frame gives
\[
\Theta_{ij,k}
=\left[\lambda_i^2+\lambda_i\lambda_j+\lambda_j^2
-\frac{5a}{4}(\lambda_i+\lambda_j)
+\frac{a^2}{4}-\frac Sn-Sx\right]h_{ijk}.
\]
Subtracting the expression with $i$ and $k$ interchanged, and using Codazzi,
we obtain
\[
(d^\nabla\Theta)_{kij}
=(\lambda_i-\lambda_k)
\left(\lambda_i+\lambda_j+\lambda_k-\frac54a\right)h_{ijk}.
\]
For the estimate of $\lambda_i-\lambda_k$ in terms of $p$, put
\[
\xi_i=\frac{\lambda_i}{\sqrt S},
\qquad
\alpha=\frac a{\sqrt S},
\qquad
\rho_i=\xi_i^2-\alpha\xi_i-\frac1n.
\]
Then $\sum_i\xi_i=0$, $\sum_i\xi_i^2=1$, and $\sum_i\xi_i^3=\alpha$.
It follows that $\rho=(\rho_i)$ satisfies
\[
\sum_i\rho_i=\sum_i\xi_i\rho_i=0,
\qquad \sum_i\rho_i^2=p.
\]
Let $e_i$ be the standard basis of $\mathbb R^n$ and $\mathbf1=(1,\ldots,1)$. Projecting $e_i+e_k$ onto $\{\mathbf1,\xi\}^{\perp}$ and applying Cauchy--Schwarz gives, for $i\ne k$,
\begin{align*}
|\rho_i+\rho_k|^2
&\le p\left(\norm{e_i+e_k}^2
-\frac{\ip{e_i+e_k}{\mathbf1}^2}{\norm{\mathbf1}^2}
-\ip{e_i+e_k}{\xi}^2\right)\\
&=p\left(2-\frac4n-(\xi_i+\xi_k)^2\right).
\end{align*}
Thus, with $u=\xi_i+\xi_k$ and $c_n^2=2-4/n$,
\begin{equation}\label{eq:rho-pair}
|\rho_i+\rho_k|^2\le p(c_n^2-u^2).
\end{equation}
The defining relation for $\rho$ also gives
\[
(\xi_i-\xi_k)^2
=2(\xi_i^2+\xi_k^2)-u^2
=\frac4n+2(\rho_i+\rho_k)+2\alpha u-u^2.
\]
Dividing by $3$, we bound the contribution of $\rho_i+\rho_k$ by \eqref{eq:rho-pair} and Young's inequality:
\[
\frac23\sqrt p\sqrt{c_n^2-u^2}
\le\frac5d p+\frac d{45}(c_n^2-u^2).
\]
Here $2\sqrt{(5/d)(d/45)}=2/3$. For the remaining terms, completing the square gives
\begin{align*}
\frac23\alpha u-\left(\frac13+\frac d{45}\right)u^2
&=\frac23\alpha u-\frac{15+d}{45}u^2\\
&=\frac{5\alpha^2}{15+d}
-\frac{15+d}{45}
\left(u-\frac{15\alpha}{15+d}\right)^2\\
&\le\frac{5\alpha^2}{15+d}
=\frac{5z}{15+d}.
\end{align*}
Together with $\alpha^2=z$, this gives
\begin{equation}\label{eq:pairwise-final}
\frac{(\lambda_i-\lambda_k)^2}{3S}
\le \frac4{3n}+\frac5d p
+\frac{(2-\frac4n)d}{45}+\frac{5z}{15+d}.
\end{equation}
It also holds for $i=k$.

To apply this estimate to $d^\nabla\Theta$, set
\[
W_{ijk}=\left(\lambda_i+\lambda_j+\lambda_k-\frac54a\right)^2h_{ijk}^2.
\]
Since $W_{ijk}$ is nonnegative and symmetric in its indices, averaging over the three pairs gives
\begin{align*}
\norm{d^\nabla\Theta}^2
&=\frac13\sum_{i,j,k}
\bigl[(\lambda_i-\lambda_j)^2+(\lambda_j-\lambda_k)^2
+(\lambda_k-\lambda_i)^2\bigr]W_{ijk}.
\end{align*}
If three real numbers are ordered as $u\ge v\ge w$, then
$(u-v)^2+(v-w)^2\le(u-w)^2$; hence the sum of their three squared pairwise differences is at most twice the largest one. Thus \eqref{eq:pairwise-final} gives
\[
\norm{d^\nabla\Theta}^2
\le
2S\left[
\frac4{3n}+\frac5d p
+\frac{(2-\frac4n)d}{45}
+\frac{5z}{15+d}
\right]
\sum_{i,j,k}W_{ijk}.
\]
By \eqref{eq:omega-r},
\[
\sum_{i,j,k}W_{ijk}=3S^3\omega_{5/4}.
\]
Consequently,
\begin{equation}\label{eq:dTheta-bound}
\frac{\norm{d^\nabla\Theta}^2}{6S^4}
\le
\left[
\frac4{3n}+\frac5d p
+\frac{(2-\frac4n)d}{45}
+\frac{5z}{15+d}
\right]\omega_{5/4}.
\end{equation}
Integrating \eqref{eq:dTheta-bound} and inserting it into \eqref{eq:T-Weitz-final} gives
\begin{align}
(1-t)T
\le{}&
\left[
\frac4{3n}
+\frac{(2-\frac4n)d}{45}
+\frac{5z}{15+d}
\right]\fint_M\omega_{5/4}\notag\\
&+\frac5d\fint_Mp\,\omega_{5/4}
+\fint_M(p-x)^2.
\label{eq:T-before-cov}
\end{align}
The variance is canceled by \eqref{eq:omega-cov}: together with \eqref{eq:omega-average} and \eqref{eq:A54}, it gives
\begin{align*}
\frac5d\fint_Mp\,\omega_{5/4}
&=\frac{5x}{d}\fint_M\omega_{5/4}
+\frac5d\fint_M(p-x)\omega_{5/4}\\
&=\frac{xA_{5/4}}d-\fint_M(p-x)^2
-\frac1{dS}\fint_M\norm{\nabla p}^2.
\end{align*}
Substituting in \eqref{eq:T-before-cov} cancels $\fint_M(p-x)^2$. Since the gradient term is nonpositive and $\fint_M\omega_{5/4}=A_{5/4}/5$, it follows that
\begin{align*}
(1-t)T
&\le A_{5/4}\left[
\frac xd+\frac4{15n}
+\frac{(2-\frac4n)d}{225}
+\frac z{15+d}\right].
\end{align*}
Dividing by $1-t$ proves the proposition.
\end{proof}

\section{The fourth-order estimate}\label{sec:fourth}

The estimate for $\Theta$ will be used in taking the trace-free part of a fourth-order tensor. Let
\[
U_{ijkl}=\frac14\bigl(h_{ijkl}+h_{jkli}+h_{klij}+h_{lijk}\bigr)
\]
be the completely symmetric part of $\nabla^2h$. For symmetric two-tensors $A,B$ write
\[
(A\odot B)_{ijkl}
=A_{ij}B_{kl}+A_{ik}B_{jl}+A_{il}B_{jk}
+A_{jk}B_{il}+A_{jl}B_{ik}+A_{kl}B_{ij},
\]
and put
\[
G_{ijkl}=g_{ij}g_{kl}+g_{ik}g_{jl}+g_{il}g_{jk}.
\]
For a completely symmetric fourth-order tensor $Q$, we use
\[
(\tr Q)_{kl}=\sum_iQ_{iikl},
\qquad
\tr^2Q=\tr(\tr Q).
\]
Write $A=\tr Q$ and $\tau=\tr^2Q$. The trace-free projection is the symmetric fourth-order case of the minimal-norm projection in \cite[Corollary~1.4]{GuoGuan}. Indeed,
\[
\tr(g\odot A)=(n+4)A+\tau g,
\qquad
\tr G=(n+2)g,
\]
so the trace-free part of $Q$ is
\begin{equation}\label{eq:GG-projection}
Q_0
=Q-\frac1{n+4}g\odot A
+\frac{\tau}{(n+4)(n+2)}G.
\end{equation}
Moreover,
\[
\ip{Q}{g\odot A}=6\norm A^2,
\qquad
\ip{Q}{G}=3\tau.
\]
Since $Q_0$ is orthogonal to every pure-trace tensor, taking the inner
product of \eqref{eq:GG-projection} with $Q$ gives
\begin{equation}\label{eq:GG-norm}
\norm{Q_0}^2
=\norm Q^2-\frac6{n+4}\norm A^2
+\frac{3}{(n+4)(n+2)}\tau^2.
\end{equation}
Polarization gives, for completely symmetric tensors $Q,R$,
\begin{equation}\label{eq:GG-polarization}
\ip{Q_0}{R_0}
=\ip QR-\frac6{n+4}\ip{\tr Q}{\tr R}
+\frac{3}{(n+4)(n+2)}(\tr^2Q)(\tr^2R).
\end{equation}

For $U$, Codazzi shows that two of the four terms in $\sum_i U_{iikl}$ equal
$\Delta h_{kl}$ and the other two vanish by minimality. Thus
$\tr U=\Delta h/2$, and Simons' equation yields
\begin{equation}\label{eq:fourth-traces}
\tr U=-\frac{S-n}{2}h=-\frac{St}{2}h,
\qquad
\tr^2U=0.
\end{equation}
The tensors to be compared with $U$ are
\[
X=h\odot\Phi,
\qquad
Y=h\odot h,
\qquad
V=X-\frac a4Y,
\qquad
V_0=(V)_0.
\]
Commutativity of $h$ and $\Phi$, together with $\tr h=\tr\Phi=\ip h\Phi=0$, gives
\[
\tr X=4h\Phi,
\quad \tr^2X=0,
\qquad
\tr Y=4h^2,
\quad \tr^2Y=4S.
\]
In particular,
\begin{equation}\label{eq:V-traces}
\tr V=4h\Phi-ah^2,
\qquad
\tr^2V=-aS.
\end{equation}
Writing $B^\circ=B-(\tr B)I/n$ for a symmetric two-tensor, we also have
\[
(\tr V)^\circ=4\Theta+4S\left(x-\frac z4\right)h.
\]
The trace of $V$ thus involves the tensor estimated in Proposition~\ref{prop:cubic}.

\begin{proposition}\label{prop:fourth}
Suppose $n<S\le2n$. With the notation above,
\begin{align}
\mathcal U
:=\fint_M\frac{\norm{U_0}^2}{S^3}
&=t(2t-1)+\frac{6t^2}{n(n+4)}+\frac32x,
\label{eq:calU}\\
\mathcal M
:=\fint_M\frac{\ip{U_0}{V_0}}{S^3}
&=\frac{12t}{5n}
+\frac{3(9n+16)}{20(n+4)}tz
+\left(\frac{3d}{5}+\frac{12t}{n+4}\right)x,
\label{eq:calM}\\
\fint_M\frac{\norm{V_0}^2}{S^3}
&=6x+\frac{24n}{n+4}
\left[\left(x-\frac z4\right)^2+T\right]
+\frac{3(n+4)}{4(n+2)}z.
\label{eq:V0norm}
\end{align}
Consequently,
\begin{equation}\label{eq:V0D}
\fint_M\frac{\norm{V_0}^2}{S^3}\le D,
\end{equation}
where
\[
D=
6x+\frac{24n}{n+4}
\left[
\left(x-\frac z4\right)^2+KA_{5/4}
\right]
+\frac{3(n+4)}{4(n+2)}z.
\]
\end{proposition}

\begin{proof}
For $\mathcal U$, the Peng--Terng identity \cite{PTseminar,PTannalen}, in the form recorded in \cite[Section~2]{CWY}, reads
\[
\norm{\nabla^2h}^2
=S(S-n)(S-2n-3)+3(\mathfrak A-2\mathfrak B).
\]
To compare $\nabla^2h$ with its completely symmetric part $U$, put
$E_{ijkl}=h_{ij,kl}-h_{ij,lk}$. Since $\nabla^2h$ is symmetric in its first three indices, expansion of its orthogonal projection onto the completely symmetric tensors gives
\[
\norm{\nabla^2h-U}^2=\frac38\sum_{i,j,k,l}E_{ijkl}^2.
\]
In a principal frame, Ricci commutation and the Gauss equation yield
\[
E_{ijkl}=(\lambda_i-\lambda_j)R_{ijkl},
\qquad
R_{ijkl}=(1+\lambda_i\lambda_j)
(\delta_{ik}\delta_{jl}-\delta_{il}\delta_{jk}).
\]
Hence
\[
\sum_{i,j,k,l}E_{ijkl}^2
=2\sum_{i\ne j}(\lambda_i-\lambda_j)^2(1+\lambda_i\lambda_j)^2.
\]
Minimality gives
\[
\sum_{i,j}(\lambda_i-\lambda_j)^2(1+\lambda_i\lambda_j)^2
=2nS-4S^2+2Sf_4-2f_3^2.
\]
Therefore
\begin{equation}\label{eq:sym-defect}
\norm{\nabla^2h-U}^2
=\frac32\bigl(Sf_4-f_3^2-2S^2+nS\bigr).
\end{equation}
This decomposition is also used in \cite[Section~3]{CWY} and \cite[Lemma~3.1]{GeTanYanZhang}. Subtracting \eqref{eq:sym-defect} from the Peng--Terng identity and using
\eqref{eq:A-2B} gives
\[
\norm U^2
=S(S-n)\left(S-2n-\frac32\right)
+\frac32\bigl(Sf_4-f_3^2-S^2\bigr).
\]
After division by $S^3$, \eqref{eq:p-formula} and
$S^{-1}=(1-t)/n$ give
\begin{align}
\frac{\norm U^2}{S^3}
&=t\left(2t-1-\frac{3}{2S}\right)
+\frac32\left(p+\frac tn\right)\notag\\
&=\frac32p+t(2t-1)+\frac{3t^2}{2n}.
\label{eq:U-pretrace}
\end{align}
Since $\tr^2U=0$, \eqref{eq:GG-norm} and \eqref{eq:fourth-traces} give
\[
\frac{\norm{U_0}^2}{S^3}
=\frac{\norm U^2}{S^3}-\frac6{n+4}\frac{\norm{\tr U}^2}{S^3}
=\frac{\norm U^2}{S^3}-\frac{3t^2}{2(n+4)}.
\]
Averaging and substituting \eqref{eq:U-pretrace} proves \eqref{eq:calU}.

For the mixed term, recall that $\diver h=\diver\Phi=0$ by the proof of Proposition~\ref{prop:cubic}. Since $U$ is completely symmetric,
\[
\ip{U}{X}
=3\bigl(h_{ij,kl}h_{ij}\Phi_{kl}
+h_{ij,kl}h_{il}\Phi_{jk}\bigr).
\]
For the first term, integration by parts and $\diver\Phi=0$ yield
\[
\fint_M h_{ij,kl}h_{ij}\Phi_{kl}
=-\fint_M h_{ij,k}h_{ij,l}\Phi_{kl}.
\]
Evaluating in a principal frame and using the symmetry of $h_{ijk}$, we obtain
\begin{align*}
\fint_M h_{ij,kl}h_{ij}\Phi_{kl}
&=-\fint_M\sum_{i,j,k}
\left(\lambda_k^2-a\lambda_k-\frac Sn\right)h_{ijk}^2\\
&=\fint_M\left(-\mathfrak A+a\mathfrak C
+\frac Sn\norm{\nabla h}^2\right).
\end{align*}
Similarly, using $\diver h=0$ and
$\Phi_{jk,l}=(\lambda_j+\lambda_k-a)h_{jkl}$,
\begin{align*}
\fint_M h_{ij,kl}h_{il}\Phi_{jk}
&=-\fint_M h_{ij,k}h_{il}\Phi_{jk,l}\\
&=-\fint_M\sum_{i,j,k}
\lambda_i(\lambda_j+\lambda_k-a)h_{ijk}^2\\
&=\fint_M(-2\mathfrak B+a\mathfrak C).
\end{align*}
By the definition of $\omega_r$ and the complete symmetry of $h_{ijk}$,
\[
\omega_0=\frac{\mathfrak A+2\mathfrak B}{S^3}.
\]
Consequently,
\begin{align*}
\fint_M\frac{\ip{U}{X}}{S^3}
&=
3\fint_M
\left[
-\frac{\mathfrak A+2\mathfrak B}{S^3}
+\frac{2a\mathfrak C}{S^3}
+\frac{\norm{\nabla h}^2}{nS^2}
\right]\\
&=
3\left[
-\fint_M\omega_0
+tz+\frac tn
\right].
\end{align*}
Using \eqref{eq:omega-average} with $r=0$,
\[
5\fint_M\omega_0=t\left(\frac1n+4z\right)-dx,
\]
we obtain
\begin{equation}\label{eq:UX}
\fint_M\frac{\ip{U}{X}}{S^3}
=\frac{12t}{5n}+\frac35tz+\frac{3d}{5}x.
\end{equation}
The other mixed term is $\ip UY=6h_{ij,kl}h_{ij}h_{kl}$. Integration by parts in $l$ gives
\[
\fint_M\ip UY
=-6\fint_M h_{ij,k}h_{ij,l}h_{kl}
=-6\fint_M\mathfrak C=-3(S-n)f_3.
\]
Since $V=X-aY/4$ and $a^2/S=z$, combining this identity with \eqref{eq:UX} yields
\begin{equation}\label{eq:UV}
\fint_M\frac{\ip{U}{V}}{S^3}
=\frac{12t}{5n}+\frac{27}{20}tz+\frac{3d}{5}x.
\end{equation}
For the trace correction, \eqref{eq:fourth-traces}, \eqref{eq:V-traces}, and $\ip{h^2}{\Phi}=\norm\Phi^2$ give
\begin{align*}
\frac{\ip{\tr U}{\tr V}}{S^3}
&=-\frac{St}{2S^3}
\left(4\ip{h}{h\Phi}-a\ip{h}{h^2}\right)\\
&=-\frac t2(4p-z),
\end{align*}
while $(\tr^2U)(\tr^2V)=0$. Hence \eqref{eq:GG-polarization} gives
\[
\frac{\ip{U_0}{V_0}}{S^3}
=\frac{\ip{U}{V}}{S^3}+\frac{3t}{n+4}(4p-z).
\]
Averaging and using \eqref{eq:UV} gives \eqref{eq:calM}, with the coefficient of $tz$ equal to $27/20-3/(n+4)=3(9n+16)/(20(n+4))$.

The computation of $\norm{V_0}^2$ uses the contraction formula
\[
\norm{A\odot B}^2
=6\norm A^2\norm B^2+6\ip AB^2+24\tr(A^2B^2).
\]
Polarizing in $B$ and using $\ip h\Phi=0$ yields
\[
\begin{aligned}
\norm X^2&=6S\norm\Phi^2+24\norm{h\Phi}^2,\\
\ip XY&=24\ip{h^3}{\Phi},
&\norm Y^2&=12S^2+24f_4.
\end{aligned}
\]
Substituting these contractions and \eqref{eq:V-traces} in \eqref{eq:GG-norm}, and using
\[
\ip{h^3}{\Phi}=\ip{h\Phi}{\Phi}+a\norm\Phi^2,
\qquad f_4=\norm\Phi^2+a^2S+\frac{S^2}{n},
\]
we obtain
\begin{align*}
\norm{V_0}^2
={}&6S\norm\Phi^2
+\frac{24n}{n+4}\left(\norm{h\Phi}^2-\frac a2\ip{h\Phi}{\Phi}\right)
-\frac{21n}{2(n+4)}a^2\norm\Phi^2\\
&+\frac{3n}{2(n+4)}a^4S
+\frac{3(n+4)}{4(n+2)}a^2S^2.
\end{align*}
Completing the square in $h\Phi$ and dividing by $S^3$ gives
\begin{equation}\label{eq:V0-pointwise}
\frac{\norm{V_0}^2}{S^3}
=6p+\frac{24n}{n+4}
\left[\frac{\norm{(h-\frac a4I)\Phi}^2}{S^3}
-\frac z2p+\frac{z^2}{16}\right]
+\frac{3(n+4)}{4(n+2)}z.
\end{equation}
Finally, $(h-aI/4)\Phi=\Theta+Sxh$ and \eqref{eq:hTheta} imply
\[
\fint_M\frac{\norm{(h-\frac a4I)\Phi}^2}{S^3}=T+x^2.
\]
Averaging \eqref{eq:V0-pointwise} proves \eqref{eq:V0norm}.
The upper bound \eqref{eq:V0D} follows from Proposition~\ref{prop:cubic}.
\end{proof}

\section{Proof of Theorem~\ref{thm:main}}\label{sec:completion}

\begin{proof}
Suppose that $n<S<2n$, so that $0<t<1/2$ and $d=3-4t>1$. The normalization $y=x/(2t)$ permits comparison with the estimates at $t=1/2$. By \eqref{eq:scalar-admissible} and the average of \eqref{eq:moment-region},
\begin{equation}\label{eq:y-region}
0\le z\le\frac{n-2}{n}<1,
\qquad
0\le y\le\frac1{2d}\left(\frac1n+\frac z4\right)
\le y_+(z):=\frac1{2n}+\frac z8.
\end{equation}
In particular, $y\le(n+2)/(8n)\le5/24<1/3$.

We first bound $K$ in \eqref{eq:Kdef} in terms of $n$ alone. The inequalities
\[
\frac1{1-t}\le2,\qquad
\frac{d}{1-t}\le3,\qquad d\ge1,\qquad 15+d\ge16
\]
imply
\[
K\le 2x+\frac z8+\frac8{15n}+\frac{2-4/n}{75}.
\]
Since $t/d\le1/2$, the scalar constraints also imply
$x-z/8\le1/(2n)$ and $2x+z\le(n-2)/n$. Thus
\[
2x+\frac z8
=\frac75\left(x-\frac z8\right)+\frac3{10}(2x+z)
\le\frac{3n+1}{10n},
\]
and hence
\begin{equation}\label{eq:K-uniform}
K\le\frac{49n+87}{150n}<C_n:=\frac{n+2}{3n}.
\end{equation}
Set $L_n=24n/(n+4)$ and $\eta=1/n+17z/48$. Since
$A_{5/4}=t(\eta-2dy)\ge0$, Proposition~\ref{prop:fourth} and \eqref{eq:K-uniform} give
\begin{equation}\label{eq:Dt}
\fint_M\frac{\norm{V_0}^2}{S^3}\le D\le D_t,
\end{equation}
where
\[
D_t=12ty+L_n\left[\left(2ty-\frac z4\right)^2
+C_nt(\eta-2dy)\right]+\frac{3(n+4)}{4(n+2)}z.
\]

With $y$ and $z$ fixed, denote the values of $\mathcal U/(4t^2)$, $\mathcal M/(2t)$, and $D_t$ at $t=1/2$ by $U_*$, $M_*$, and $D_*$. Explicitly,
\begin{align*}
U_*&=\frac32y+\frac{3}{2n(n+4)},\\
M_*&=\frac6{5n}+\frac{3(9n+16)}{40(n+4)}z
+\frac{3(n+14)}{5(n+4)}y,\\
D_*&=6y+L_n\left[\left(y-\frac z4\right)^2
+C_n\left(\frac\eta2-y\right)\right]
+\frac{3(n+4)}{4(n+2)}z.
\end{align*}
Substituting $x=2ty$ into \eqref{eq:calU} and \eqref{eq:calM}, we find
\begin{align}
\frac{\mathcal U}{4t^2}-U_*
&=\frac{1-2t}{4t}(3y-1)<0,\label{eq:Ucompare}\\
\frac{\mathcal M}{2t}-M_*
&=\frac{6(n-1)}{5(n+4)}(1-2t)y\ge0.\label{eq:Mcompare}
\end{align}
For $D_t$, subtraction gives
\[
\frac{D_*-D_t}{1-2t}
=6y+L_n\left[(1+2t)y^2-\frac{yz}{2}
+C_n\left(\frac\eta2+(4t-1)y\right)\right].
\]
Here \eqref{eq:y-region} implies
\[
\frac\eta2-dy\ge\frac{5z}{96},
\qquad
\frac\eta2+(4t-1)y=\frac\eta2-dy+2y\ge2y.
\]
Since $2C_n-z/2>1/6$, it follows that
\[
\frac{D_*-D_t}{1-2t}
\ge6y+L_n\left[(1+2t)y^2+\left(2C_n-\frac z2\right)y\right]\ge0.
\]
Thus $D_t\le D_*$. Applying \eqref{eq:Dt} and \eqref{eq:Ucompare}--\eqref{eq:Mcompare} to $V_0-(2+z)U_0/(2t)$, we obtain
\begin{align}
0&\le\fint_M\frac{\left|V_0-\frac{2+z}{2t}U_0\right|^2}{S^3}\notag\\
&\le D_t-2(2+z)\frac{\mathcal M}{2t}
+(2+z)^2\frac{\mathcal U}{4t^2}\notag\\
&<D_*-2(2+z)M_*+(2+z)^2U_*=:Q(y,z).
\label{eq:square-final}
\end{align}
Strictness follows from \eqref{eq:Ucompare} and $2+z>0$.

For fixed $z$, $Q$ is convex in $y$, since $Q_{yy}=48n/(n+4)>0$. It therefore suffices to check the endpoints of $[0,y_+(z)]$. At $y=0$,
\[
Q(0,z)=-\frac{2(2n+13)}{5n(n+4)}-b_nz+a_nz^2,
\]
where
\[
a_n=\frac{3(n^2-16n+10)}{20n(n+4)},
\qquad
b_n=\frac{8n^3+14n^2+5n+108}{15n(n+2)(n+4)}.
\]
For $n\ge3$ one has $a_n\le3/20$, while
\[
b_n\ge\frac1{15}\frac{n}{n+2}\frac{8n+14}{n+4}\ge\frac15.
\]
Using $z^2\le z$, we obtain
\begin{equation}\label{eq:Q-left-negative}
Q(0,z)\le-\frac{2(2n+13)}{5n(n+4)}-\frac z{20}<0.
\end{equation}
At $y=y_+(z)$, substitution gives
\[
Q(y_+(z),z)
=z\left[\frac3{16}z^2
-\frac{3(n^2+2n-12)}{8n(n+4)}z
-\frac{(n-1)(n+5)}{3(n+2)(n+4)}\right].
\]
The coefficient $n^2+2n-12$ is positive for $n\ge3$, and
\[
\frac{(n-1)(n+5)}{3(n+2)(n+4)}
>\frac{n-1}{3(n+2)}
>\frac{n-2}{4n}\ge\frac z4;
\]
the middle inequality follows from
$4n(n-1)-3(n+2)(n-2)=(n-2)^2+8>0$.
Consequently,
\begin{equation}\label{eq:Q-right-negative}
Q(y_+(z),z)\le z\left(\frac3{16}z^2-\frac z4\right)
\le-\frac{z^2}{16}\le0.
\end{equation}
Writing $y=\theta y_+(z)$, $0\le\theta\le1$, convexity now gives
\[
Q(y,z)\le(1-\theta)Q(0,z)+\theta Q(y_+(z),z)\le0,
\]
contrary to \eqref{eq:square-final}. This excludes $n<S<2n$.

Finally, suppose $S=2n$. Then $t=1/2$, $d=1$, and $y=x$. The constraints \eqref{eq:y-region} and the estimates through \eqref{eq:Dt} still hold, and
$\mathcal U=U_*$, $\mathcal M=M_*$, and $D_t=D_*$. Thus
\begin{equation}\label{eq:equality-square}
0\le\fint_M\frac{\norm{V_0-(2+z)U_0}^2}{S^3}\le Q(y,z).
\end{equation}
If $z>0$, both endpoint values of $Q$ are negative by
\eqref{eq:Q-left-negative} and \eqref{eq:Q-right-negative}, contradicting
\eqref{eq:equality-square}. Hence $z=0$ and $f_3=0$.
For $z=0$, the left endpoint value is negative and the right endpoint value is zero. By convexity, $y=y_+(0)$ is the only possibility consistent with \eqref{eq:equality-square}. Thus
\[
x=y=y_+(0)=\frac1{2n},
\qquad A_{5/4}=t\left(\frac1n+\frac{17z}{48}\right)-dx=0.
\]
Proposition~\ref{prop:cubic} gives $T=0$. The integrand is nonnegative and continuous, so $\Theta$ vanishes on every component. With $a=0$, $S/n=2$, and $Sx=1$, we obtain
\[
0=\Theta=h(h^2-2I)-h=h^3-3h.
\]
Each principal curvature is one of $-\sqrt3,0,\sqrt3$.
If the respective multiplicities are $m_-,m_0,m_+$, minimality and
$S=2n$ imply $m_+=m_-$ and $3(m_++m_-)=2n$; hence all three
multiplicities equal $n/3$. Each connected component is thus isoparametric
with three principal curvatures. Cartan's classification
\cite{Cartan1939,Cecil2008} gives $n/3\in\{1,2,4,8\}$ and local congruence
to the corresponding minimal Cartan hypersurface. These hypersurfaces have $S=2n$, which also proves the converse.
\end{proof}

\raggedbottom

\end{document}